\documentclass[letter,12pt]{article}

\usepackage{amssymb}
\usepackage{amsthm}
\usepackage{amsfonts}
\usepackage[latin2]{inputenc}
\usepackage{amsmath}
\usepackage{anysize}
\usepackage{hyperref}
\usepackage{bbm}

\marginsize{2.6cm}{2.6cm}{1cm}{2cm}

\newtheorem*{theorem*}{Theorem}
\newtheorem{main_theorem}{Theorem}
\newtheorem{theorem}{Theorem}[section]

\newtheorem{proposition}[theorem]{Proposition}

\def\del{\partial}
\def\dbar{\bar\partial}
\def\ddbar{\del\dbar}

\def\del{\partial}

\def\o{\omega}

\title{Geodesic Rays and K\"ahler--Ricci Trajectories on Fano Manifolds}

\author{Tam\'as Darvas\thanks{Research supported by BSF grant 2012236. } \ and Weiyong He\thanks{Research supported partially by NSF grant 1005392.\newline  2010 Mathematics subject classification 53C55, 32W20, 32U05.} }
\date{\vspace{-0.3in}}

\begin{document}
\maketitle
\begin{abstract}
Suppose $(X,J,\o)$ is a Fano manifold and $t \to r_t$ is a diverging K\"ahler-Ricci trajectory. We construct a bounded geodesic ray $t \to u_t$ weakly asymptotic to $t \to r_t$, along which Ding's $\mathcal F$--functional  decreases, partially confirming a folklore conjecture. In absence of non-trivial holomorphic vector fields this proves the equivalence between geodesic stability of the $\mathcal F$--functional and existence of K\"ahler--Einstein metrics. We also explore applications of our construction to Tian's $\alpha$--invariant.
\end{abstract}
\section{Introduction and Main Results}

Let $(X,J,\o)$ be a compact connected Fano K\"ahler manifold  normalized by $c_1(X)=[\o]_{dR}$. If $\o'$ is another K\"ahler metric on $X$ satisfying $[\o']_{dR}=[\o]_{dR}$, by the $\ddbar$--lemma of Hodge theory 
there exists a potential $\varphi \in C^\infty(X)$ such that
$$\o' = \o + i\ddbar \varphi,$$
and up to a constant $\varphi$ uniquely determines $\o'$. Hence, one can study K\"ahler metrics in the cohomology class of $\o$ by studying certain smooth functions. This motivates the introduction of the space of smooth K\"ahler potentials:
$$\mathcal H = \{ u \in C^\infty(X) | \ \o_u:=\o + i\partial \bar \partial u > 0\}.$$
Clearly, $\mathcal H$ is a Fr\'{e}chet manifold as an open subset of $C^{\infty}(X)$, so for $v \in \mathcal{H}$ one can identify $T_v \mathcal{H}$ with $C^{\infty}(X)$. Given $1 \leq p < \infty$, we introduce  $L^p$ type Finsler-metrics on $\mathcal H$:
\begin{equation}\label{FinslerDef}
\|\xi\|_{p,v}= \Big( \frac{1}{\textup{Vol}(X)}\int_X |\xi|^p \o_v^n \Big)^{\frac{1}{p}}, \ \ \ \ \xi \in T_v \mathcal{H},
\end{equation}
where $\textup{Vol}(X)=\int_X \o^n$ is an invariant of the class $\mathcal H$.  When $p=2$ we obtain the much studied Mabuchi Riemannian structure initially investigated in \cite{m,s,do} in connection with special K\"ahler metrics. As pointed out in \cite{da4}, in the case $p=1$ one recovers the strong topology/geometry of $\mathcal H$, as introduced in \cite{bbegz}, which proved to be extremely useful in the study of weak solutions to complex Monge--Amp\`ere equations. In considering the general case $p\geq 1$, we hope to unify the treatment of these two motivating examples.

Even more general Orlicz--Finsler structures were studied in \cite{da4} and we recall some of the notations and results of this paper before we state our main theorems. A curve $[0,1]\ni t \to \alpha_t \in \mathcal{H}$ is smooth if the function $\alpha(t,x)=\alpha_t(x) \in C^\infty([0,1] \times X)$. As usual, the length of a smooth curve $t \to \alpha_t$ is computed by the formula:
\begin{equation}\label{curve_length_def}
l_p(\alpha)=\int_0^1\|\dot \alpha_t\|_{p,\alpha_t}dt.
\end{equation}
The path length distance $d_p({u_0},{u_1})$ between ${u_0},{u_1} \in \mathcal H$ is the infimum of the length of smooth curves joining $u_0,u_1$. In \cite{da4} it is proved that $d_p(u_0,u_1)=0$ if and only if $u_0 = u_1$, thus $(\mathcal H, d_p)$ is a metric space, which is a generalization of a result of X.X. Chen in the case $p=2$ \cite{c}.

Let us recall some facts about the Riemannian case $p=2$. We will be very brief and for details we refer to the recent survey \cite{bl2}. In this case one can compute the the associated Levi-Civita connection $\nabla_{(\cdot)}(\cdot)$ and curvature tensor which is non-positive.

Suppose $S = \{ 0 < \textup{Re }s < 1\}\subset \Bbb C$. Following \cite{s}, one can argue that a smooth curve $[0,1] \ni t \to u_t \in \mathcal H$ connecting $u_0,u_1 \in \mathcal H$ is a Riemannian geodesic ($\nabla_{\dot u_t}\dot u_t=0$) if its complexification $u(s,x)=u_{\textup{Re }s}(x)$ is the (unique) smooth solution of the  following Dirichlet problem on $S \times X$:
\begin{alignat}{2}\label{BVPGeod}
&(\pi^* \o + i \partial \overline{\partial}u)^{n+1}=0 \nonumber\\
&u(t+ir,x) =u(t,x) \ \forall x \in X, t \in (0,1), r \in \Bbb R \\
&u(0,x)=u_0(x), u(1,x)=u_1(x), \ x \in X.\nonumber
\end{alignat}
Unfortunately, the above problem does not have smooth solutions (see \cite{lv,da1}), but a unique solution in the sense of Bedford-Taylor does exist such that $i\partial \bar \partial u$ has bounded coefficients (see \cite{c} with complements in \cite{bl}). The most general result about regularity was proved in \cite{bd,brm2} (see \cite{h1} for a different approach) but regularity higher then $C^{1,1}$ is not possible by examples provided in \cite{dl}. The resulting curve
$$[0,1] \ni t \to u_t \in \mathcal H_{\Delta} = \{ \Delta u \in L^\infty, \ \o + i\ddbar u \geq 0\}$$
is called the weak geodesic joining $u_0,u_1$. As we just explained, this curve leaves the space $\mathcal H$, hence it cannot be a Riemannian geodesic, but as argued in \cite{da4}, it interacts well with all the path length metrics $d_p$, i.e.
\begin{equation}\label{distgeod}
d_p(u_0,u_1)=\|\dot u_t \|_{p,u_t}, \ t \in [0,1], p \geq 1.
\end{equation}
In fact, $t \to u_t$ is an actual $d_p$--metric geodesic joining $u_0,u_1$ in the metric completion $\overline{(\mathcal H, d_p)}=(\mathcal E^p(X,\o),d_p)$ as we recall now.

The set of $\o$--plurisubharmonic functions is the following class:
$$\textup{PSH}(X,\o)=\{ u \in L^1(X), \textup{ u is u.s.c. and } \o + i\ddbar u\geq 0 \}.$$
If $u \in \textup{PSH}(X,\o)$, as explained in  \cite{gz}, one can define the non-pluripolar measure $\o_u^n$ that coincides with the usual Bedford-Taylor volume when $u$ is bounded. We say that $\o_u^n$ has full volume ($u \in \mathcal E(X,\omega)$) if $\int_X \o_u^n = \int_X \o^n$. Given  $v \in \mathcal E(X,\omega)$, we say that $v \in \mathcal E^p(X,\omega)$ if
$$\int_X  |v|^p \o_v^n < \infty,$$
The following trivial inclusion will be essential to us later:
\begin{equation}\label{trivinclusion}
\mathcal H_0 = \textup{PSH}(X,\o) \cap L^\infty(X) \subset \bigcap_{p \geq 1}\mathcal E^p(X,\o).
\end{equation}
For a quick review of finite energy classes $\mathcal E^p(X,\o)$ we refer to \cite[Section 2.3]{da3}. Next we recall the induced geodesic metric space structure on $\mathcal E^p(X,\o)$. Suppose $u_0,u_1 \in \mathcal E^p(X,\omega)$.  Let $\{u^k_0\}_{k \in \Bbb N},\{u^k_1 \}_{k \in \Bbb N}\subset \mathcal H$ be sequences decreasing pointwise to $u_0$ and $u_1$ respectively. By \cite{bk,de} it is always possible to find such approximating sequences. We define the metric $d_p(u_0,u_1)$ as follows:
\begin{equation}\label{EpDistDef}
d_p(u_0,u_1) = \lim_{k \to \infty} d_p(u^k_0,u^k_1).
\end{equation}
As justified in \cite[Theorem 2]{da4} the above limit exists is well defined and defines a metric on $\mathcal E^p(X,\o)$.

Let us also define geodesics in this space. Recall that by a $\rho$\emph{--geodesic} in a metric space $(M,\rho)$ we understand a curve $[a,b] \ni t \to g_t \in M$ for which there exists $C>0$ satisfying:
$$\rho(g_{t_1},g_{t_2}) = C|t_1 - t_2|, \ t_1,t_2 \in [a,b].$$
Let $u^k_t : [0,1] \to \mathcal H_{\Delta}$ be the weak geodesic joining $u^k_0,u^k_1$. We define $t\to u_t$ as the decreasing limit:
\begin{equation}\label{EpGeodDef}
u_t = \lim_{k \to + \infty}u^k_t, \ t \in (0,1).
\end{equation}
The curve $t \to u_t$ is well defined and  $u_t \in \mathcal E^p(X,\omega), \ t \in (0,1)$, as follows from the results of  \cite{da3}. By \cite[Theorem 2]{da4} this curve is a $d_p$--geodesic joining $u_0,u_1$ and we have
\begin{equation}\label{CompletionEq}
\overline{(\mathcal H,d_p)} = (\mathcal E^p(X,\o),d_p), \ \ p \geq 1.
\end{equation}

Functionals play an important role in the investigation of special K\"ahler metrics. Recall that the Aubin-Mabuchi energy and Ding's $\mathcal F$--functional are defined as follows:
\begin{equation}\label{AMdef}
AM(v)=\frac{1}{(n+1)\textup{Vol}(X)}\sum_{j=0}^n\int_{X} v\o^j\wedge (\o + i\partial\bar\partial v)^{n-j},
\end{equation}
\begin{equation}\label{Dingdef}
\mathcal F(v) = - AM(v) - \log \int_X e^{-v + f_\o}\o^n,
\end{equation}
where $v \in \mathcal H$ and $f_\o \in C^\infty(X)$ is the Ricci potential of $\o$, i.e. $\textup{Ric }\o =\o+i\partial\bar\partial f_\o$ normalized by
$\int_X e^{f_\o}\o^n =1.$ It was argued in \cite{da4} that both of these functionals are continuous with respect to all metrics $d_p$, hence extend to $\mathcal E^p(X,\o)$ continuously. Also, $AM$ is linear along the geodesics defined in \eqref{EpGeodDef}, whereas $\mathcal F$ is convex. As the map $u \to \o_u$ is translation invariant, one may want normalize K\"ahler potentials to obtain an equivalence between metrics and potentials. This can be done by only considering potentials from the "totally geodesic" hypersurfaces  $$\mathcal H_{AM} = \mathcal H \cap \{ AM(\cdot)=0\},$$
$$\mathcal H_{0,AM} = L^\infty(X) \cap \textup{PSH}(X,\o) \cap \{ AM(\cdot)=0\},$$
$$\mathcal E^p_{AM}(X,\o) = \mathcal E^p(X,\o) \cap \{ AM(\cdot)=0\}.$$
A smooth metric $\omega_{u_{KE}}$ is \emph{K\"ahler-Einstein} if $$\omega_{u_{KE}}=\textup{Ric }\omega_{u_{KE}}.$$
One can study such metrics by looking at the long time asymptotics of the Hamilton's K\"ahler--Ricci flow:
\begin{equation}\label{RicciflowEq}
\begin{cases}
\frac{d \o_{r_t}}{dt} = - \textup{Ric }\o_{r_t} + \o_{r_t}, \\
r_0 = v. \end{cases}
\end{equation}
As proved in \cite{ca}, for any $v \in \mathcal H_{AM}$, this problem has a smooth solution $$[0,1) \ni t \to r_t \in \mathcal H_{AM}.$$ It follows from a theorem of Perelman and work of Chen-Tian, Tian-Zhu and Phong-Song-Sturm-Weinkove,
that whenever a K\"ahler--Einstein metric cohomologous to $\o$ exists, then $\o_{r_t}$ converges exponentially fast to one such metric (see \cite{ct}, \cite{tz}, \cite{pssw}).

We remark that our choice of normalization is different from the alternatives used in the literature (see \cite[Chapter 6]{beg}). We choose to work with the normalization $AM(\cdot)=0$, as this seems to be the most natural one from the point of view of Mabuchi geometry. Indeed, that Aubin-Mabuchi energy is continuous with respect to all metrics $d_p$ and is linear along the geodesic segments defined in \eqref{EpGeodDef}. It will require some careful analysis, but as we shall see, from the point of view of long time asymptotics, this normalization is equivalent to other alternatives.

Suppose $(M,\rho)$ is a geodesic metric space and $[0,\infty) \ni t\to c_t \in \mathcal M$ is a continuous curve. We say that the unit speed $\rho$--geodesic ray $[0,\infty) \ni t\to g_t \in \mathcal M$ is \emph{weakly asymptotic} to the curve $t \to c_t$, if there exists $t_j \to \infty$ and unit speed $\rho$--geodesic segments $[0,\rho(c_0,c_{t_j})] \ni t\to g^j_t \in \mathcal M$ connecting $c_0$ and $c_{t_j}$ such that
$$\lim_{j \to\infty}\rho(g^j_t,g_t)=0, \ t \in [0,\infty).$$

We clearly need $\lim_j \rho(c_0,c_{t_j})=\infty$ in this last definition, hence to construct $d_p$--geodesic rays weakly asymptotic to diverging K\"ahler-Ricci trajectories, we first need to prove the following result, which improves on the main result of \cite{mc} and partly generalizes \cite[Theorem 6]{da4}. For a similar result about the Calabi metric we refer to \cite{cr}.

\begin{main_theorem} \label{RicciBounded}Suppose $(X,J,\o)$ is a Fano manifold and $p \geq 1$. There exists a K\"ahler--Einstein metric in $\mathcal H$ if and only if every K\"ahler--Ricci flow trajectory $[0,\infty) \ni t \to r_t \in \mathcal H_{AM}$ is $d_p$--bounded. More precisely, the $C^0$ bound along the flow is equivalent to the $d_p$ bound:
$$\frac{1}{C}d_p(r_0,r_t) - C\leq \sup_X |r_t| \leq C d_p(r_0,r_t) + C,$$
for some $C(p,r)>1$.
\end{main_theorem}

Using this theorem, the recently established convexity of the K-energy functional from \cite{bb} (for a different approach see \cite{clp}), the compactness theorem of \cite{bbegz}, and the divergence analysis of K\"ahler-Ricci trajectories from \cite{r1}, we establish our main result:

\begin{main_theorem} \label{geodconstruct}Suppose $(X,J,\o)$ is a Fano manifold without a K\"ahler--Einstein metric in $\mathcal H$ and $[0,\infty) \ni t \to r_t \in \mathcal H_{AM}$ is a K\"ahler-Ricci trajectory. Then there exists a curve $[0,\infty) \ni t \to u_t \in \mathcal H_{0,AM}$ which is a $d_p$--geodesic ray weakly asymptotic to $t \to r_t$ for all $p \geq 1$. In addition to this, $t \to u_t$ satisfies the following:
\begin{itemize}
\item[(i)] $t \to \mathcal F(u_t)$ is decreasing,
\item[(ii)] the "sup-normalized" potentials $u_t - \sup_X (u_t - u_0) \in \mathcal H_0$ decrease pointwise to $u_\infty \in \textup{PSH}(X,\o)$ for which $\int_X e^{-\frac{n}{n+1}u_\infty}\o^n =\infty$.
\end{itemize}
If additionally $(X,J)$ does not admit non--trivial holomorphic vector fields, then $t \to \mathcal F(u_t)$ is strictly decreasing.
\end{main_theorem}

We note that the normalizing condition $AM(u_t)=0$ in the above result assures that geodesic ray $t \to u_t$ is  \emph{non--trivial}, i.e. $u_t \neq u_0+ ct$.

This theorem provides a partial answer to a folklore conjecture, perhaps first suggested by \cite{lnt}, which says that one should be able to construct "destabilizing" geodesic rays asymptotic to diverging K\"ahler-Ricci trajectories. For a precise statement and connections with other results we refer to \cite[Conjecture 4.10]{r1}.

Given their connection with special K\"ahler metrics, constructing geodesic rays in the space of K\"ahler potentials from geometric data has drawn a lot of interest. We mention \cite{ph1,ph2}, where the authors constructed rays out of algebraic test configurations. The work \cite{rwn}, builds on this and constructs rays out of more general analytic test configurations via their Legendre transform. For related results we also mention \cite{at, ct2, sz,rz} in a fast expanding literature. Perhaps one of the advantages of our method is that the ray we construct instantly gives geometric information about special K\"ahler metrics without further results, as it will be evidenced in Theorem 3 below.

We hope that the methods developed here will be the building blocks of future results constructing geodesic rays asymptotic to different (geometric) flow trajectories. Motivated by this we prove a very general result in Theorem \ref{curvegeod} from which Theorem 2 will follow.

On Fano manifolds not admitting K\"ahler-Einstein metrics, part (ii) of Theorem 2 ensures the bound $\alpha(X) \leq n/(n+1)$ for Tian's alpha invariant:
$$\alpha(X)=\sup \Big\{ \alpha , \int_X e^{-\alpha(u-\sup_X u)}\o^n \leq C_\alpha <+\infty, \ u \in \textup{PSH}(X,\o) \Big\}.$$
This is a well known result of Tian \cite{t1}. The fact that the geodesic ray $t \to u_t$ is able to detect a potential $u_\infty$ satisfying $\int_X e^{-\frac{n}{n+1}u}\o^n =\infty$, is analogous to the main result of \cite{r1}, where it is shown that one can find such potential using a subsequence of metrics along a diverging  K\"ahler-Ricci trajectory. We refer to this paper for relations with Nadel sheaves.

It would be interesting to see if a geodesic ray produced by the above theorem is in fact unique. We prove that this ray is bounded, but it is not clear if this curve has more regularity. Finally, we believe that $t \to \mathcal F(u_t)$ is strictly decreasing regardless whether $(X,J)$ admits non--trivial holomorphic vector fields or not and prove this in the case when the Futaki invariant is non--zero (Proposition \ref{FutakiStrict}).

Lastly, we note the following theorem, which is a consequence of the previous result, and in the case $p=2$ gives the K\"ahler-Einstein analog of Donaldson's conjectures on existence of constant scalar curvature metrics \cite{do,h2}:

\begin{main_theorem} \label{raystability}Suppose $p \in \{1,2 \}$ and $(X,J,\o)$ is a Fano manifold without non--trivial holomorphic vector--fields and $u \in \mathcal H$. There exists no K\"ahler-Einstein metric in $\mathcal H$ if and only if for any $u \in \mathcal H$ there exists a $d_p$--geodesic ray $[0,\infty) \ni t \to u_t \in \mathcal H_{0,AM}$ with $u_0=u$ such that the function $t \to \mathcal F(u_t)$ is strictly decreasing.
\end{main_theorem}

As pointed out to us by R. Berman, for $p=2$ this last theorem follows from \cite{brm1} and the recently established equivalence between K--stability and existence of K\"ahler--Einstein metrics. Our approach however is purely analytical and avoids the use of algebro--geometric techniques.

Although we do not pursue such generality, we remark that Theorem 1 and Theorem 2 also hold for the very general Orlicz-Finsler structures $(\mathcal H,d_\chi)$ studied in \cite{da4}.

\section{Preliminaries}

\subsection{The Metric Spaces $(\mathcal H,d_p)$} In this short paragraph we further elaborate on the metric spaces $(\mathcal H,d_p)$. By the definition, we have the inclusion
$\mathcal E^p(X,\o) \subset \mathcal E^{p'}(X,\o),$
for $p' \leq p$ and also the metric $d_p$ dominates $d_{p'}$. What is more, it follows that for $u_0,u_1 \in \mathcal E^p(X,\o)$, the curve defined in $\eqref{EpGeodDef}$ is a geodesic with respect to both $d_p$ and $d_p'$ (perhaps of different length).
Using this and \eqref{trivinclusion} we can conclude the following crucial observation.
\begin{proposition}\label{samegeod}For $u_0,u_1 \in \mathcal H_0$, the curve $[0,1] \ni t \to u_t \in \mathcal H_0$ from \eqref{EpGeodDef} will be a $d_p$--geodesic joining $u_0,u_1$ for all $p \geq 1$.
\end{proposition}

We note that for $p \neq 2$, the $d_p$--geodesic connecting $u_0,u_1$ may not be unique. See \cite{da4} for examples of $d_1$--geodesic segments that are different from the ones defined in \eqref{EpGeodDef}.

In hopes of characterizing convergence in the metric completion $(\mathcal E^p(X,\o),d_p)$ more explicitly, for $u_0,u_1 \in \mathcal E^p(X,\o)$ one can introduce the following functional (see \cite{da4,g}):
$$I_p(u_0,u_1) = \Big(\int_X |u_0 - u_1|^p \o_{u_0}^n\Big)^{1/p} + \Big(\int_X|u_0 - u_1|^p \o_{u_1}^n\Big)^{1/p}. $$
In \cite[Theorem 3]{da4} it is proved that there exists $C(p) >1$ such that
\begin{equation}\label{mabenergeqv}\frac{1}{C}I_p(u_0,u_1) \leq d_p(u_0,u_1) \leq C I_p(u_0 , u_1).
\end{equation}
This double estimate implies that there exists $C(p) > 1$ such that
\begin{equation}\label{supmabest}
\sup_X u \leq C d_p(u,0) + C.
\end{equation}
Also, if $d_p(u_k, u) \to 0$ then $u_k \to u$ a.e. and also $\o_{u_k}^n \to \o^n_u$ weakly.  For more details we refer to \cite[Theorems 3-6]{da4}. Our first observation says that in the presence of uniform $C^0$--estimates all the $d_p$ geometries are equivalent.
\begin{proposition} \label{boundedconvergence} Suppose $\{u_k\}_{k \in \Bbb N} \subset \mathcal H_0 = \textup{PSH}(X,\o) \cap L^\infty$  and  $\| u_k\|_{L^\infty} \leq D$ for some $D >0$. Then $\{u_k\}_{k \in \Bbb N}$ is $d_p$--Cauchy if and only if it is $d_1$--Cauchy. If this condition holds then in addition the limit $u = \lim_k u_k$ also satisfies $\| u\|_{L^\infty} \leq D$.
\end{proposition}

\begin{proof} The equivalence follows from \eqref{mabenergeqv} and basic facts about $L^p$ norms. The estimate $\| u\|_{L^\infty} \leq D$ also follows, as from \cite[Theorem 5(i)]{da4} we have $u_k \to u$ in capacity, hence $u_k \to u$ pointwise a.e..
\end{proof}

We recall the compactness theorem \cite[Theorem 2.17]{bbegz}. Before we write down the statement, let us first recall the notion of strong convergence and and entropy.
As introduced in \cite{bbegz}, we say that a sequence $u_k \in \mathcal H$ converges \emph{strongly} to $u \in \mathcal E^1(X,\o)$ if $u_k \to_{L^1} u$ and $AM(u_k) \to AM(u)$. As argued in \cite[Proposition 5.9]{da4}, one has $u_k \to u$ strongly if and only if $I_1(u_k,u) \to 0$, which in turn is equivalent to $d_1(u_k,u) \to 0$ according to \eqref{mabenergeqv}. The Mabuchi \emph{K-energy} functional
$\mathcal M: \mathcal H \to \Bbb R$, which will be used by us later, is given by the following formula:
$$\mathcal M(u) = n AM(u) - L(u) + H_{\o}(\o_u),$$
where $H_{\o}(\o_u) = \int_X \log(\o_u^n/\o^n)\o_u^n$ is the \emph{entropy} of $\o_u^n$ with respect to $\o^n$ and $L(u)$ is the following operator:
$$L(u) = \sum_{j=0}^{n-1}\int_X u \ {\textup{Ric } \o} \wedge \o_u^{j} \wedge \o^{n-1-j}.$$
In the presence of bounded entropy the following compactness result holds:
\begin{proposition}\textup{\cite[Proposition 2.6, Theorem 2.17]{bbegz}} \label{bbegzcomp} Suppose $\{ u_k\}_{k \in \Bbb N} \subset \mathcal H$ is such that $|\sup_X u_k|, H_{\o}(\o_{u_k}) \leq D$ for some $D \geq0$. Then there exists $u \in \mathcal E^1(X,\o)$ and $k_l \to \infty$ such that $\lim_{l \to \infty} d_1(u_{k_l},u)=0$.
\end{proposition}

Putting together the last two results we can write:

\begin{theorem}\label{boundedcomp} Suppose $\{ u_k\}_{k \in \Bbb N} \subset \mathcal H$ is such that $H_{\o}(\o_{u_k}), \| u_k\|_{L^\infty} \leq D$ for some $D \geq0$. Then there exists $u \in \mathcal H_0$ with $\| u\|_{L^\infty} \leq D$ and $k_l \to \infty$ such that  $d_p(u_{k_l},u) \to 0$ for all $p \geq 1$.
\end{theorem}

In our computations we will need the following bound for the $L$ functional in the expression of the Mabuchi K-energy:

\begin{proposition}\label{LEst} For any $p \geq 1$ there exists $C(p) >1$ such that
$$|L(u)| \leq C d_p(0,u), \ u \in \mathcal H.$$
\end{proposition}

\begin{proof} There exists $C > 0$ such that $\textup{Ric } \o \leq C \o$. We can start writing:
\begin{flalign*}
|L(u)| & \leq C \sum_{j=1}^n \int_X |u| \o^{j} \wedge \o^{n-j}_u \\
&\leq C \int_X \Big|\frac{u}{2}\Big| \o_{\frac{u}{2}}^n \\
&\leq C \Big(\int_X \Big|\frac{u}{2}\Big|^p \o_{\frac{u}{2}}^n\Big)^{1/p} \\
&\leq C d_p\Big(0,\frac{u}{2}\Big)\\
& \leq C d_p(0,u),
\end{flalign*}
where in the penultimate inequality we have used \eqref{mabenergeqv} and in the last inequality we have used \cite[Lemma 5.3]{da4}.
\end{proof}

Finally, we recall a result about bounded geodesics which will be very useful to us later:

\begin{theorem}\textup{\cite[Theorem 1]{da2}} \label{m_norm_thm} Given a bounded weak geodesic $[0,1] \ni t  \to u_t \in \mathcal H_0$ connecting $u_0,u_1 \in \mathcal H_0$, i.e. a bounded solution to \eqref{BVPGeod}, there exists $M_u,m_u \in \Bbb R$ such that for any $a,b \in [0,1]$ we have
\begin{enumerate}
\item[(i)] $\inf_{X}\frac{u_a-u_b}{a-b}= m_u,$
\item[(ii)] $\sup_{X}\frac{u_a-u_b}{a-b}=M_u.$
\end{enumerate}
\end{theorem}

This result tells us that for a bounded weak geodesic $[0,1] \ni t \to u_t \in \mathcal H_0$, which is also a $d_p$--geodesic for all $p \geq 1$, the function $t \to \sup_X(u_t - u_0)$ is linear. As explained in the introduction of \cite{da2}, this implies that $t \to \tilde u_t = u_t - \sup_X(u_t - u_0) \in \mathcal H$ is a geodesic that is decreasing in $t$ (one can see that $\dot {\tilde u}_t \leq 0$). Since $\sup_X \tilde u_t$ is bounded, the pointwise limit $u_\infty = \lim_{t \to \infty} \tilde u_t$ is different from $-\infty.$ As we shall see by the end of this paper, for certain geodesic rays one can draw geometric conclusions about the manifolds $X$, by studying the singularity type of $u_\infty$.

\subsection{Diverging K\"ahler-Ricci Trajectories}

In this short paragraph we recall estimates along diverging K\"ahler Ricci trajectories, that will allow us to apply Theorem \ref{boundedcomp} along such curves. Unfortunately most of the literature on the K\"ahler-Ricci flow uses a normalization different from ours (see below). We will argue that the most important estimates have analogs for our $AM$--normalized trajectories as well.

It is well known that flow equation \eqref{RicciflowEq}
can be rewritten as the scalar equation $$\o_{r_t}^n = e^{f_\o - r_t + \dot r_t + \beta(t)}\o^n,$$
where $\beta: [0,\infty) \to \Bbb R$ is a function chosen depending on the desired normalization condition on $r_t$. In our investigations we  will use the normalizing condition $AM(r_t)=0$. However most of the literature on the K\"ahler-Ricci flow uses the normalization $t \to \tilde r_t$ for which $\beta(t)=0$  and ${\tilde r}_0 = v + c$, with $c$ carefully chosen (see \cite[(2.10)]{pss}). Evidently, in this latter case the scalar equation becomes
\begin{equation}\label{RicciflowEqScalar}
\o_{{\tilde r}_t}^n = e^{- {\tilde r}_t + \dot {\tilde r}_t + f_\o}\o^n,
\end{equation}
and the conversion from this normalization to the one employed by us is given by the formula
$$r_t = {\tilde r}_t - AM({\tilde r}_t), \ t \geq 0.$$
The following result brings together estimates for the trajectory $t \to {\tilde r}_t$ that we will need. Most  of these are classical and well known, for the others sketch the proof:

\begin{proposition} \label{R2NormEst}Suppose $t \to {\tilde r}_t$ is a K\"ahler-Ricci trajectory normalized according to \eqref{RicciflowEqScalar}. For any $t \geq 0$ we have:
\begin{itemize}
\item[(i)] $\| \dot {\tilde r}_t\|_{L^\infty}, \| f_{\o_t}\|_{L^\infty} \leq C$ for some $C > 1$.
\item[(ii)] $-C \leq AM({\tilde r}_t)$, in particular $-\int_X {\tilde r}_t \o_{{\tilde r}_t}^n \leq n\int_X {\tilde r}_t \o^n + C$ for some $C > 1$.
\item[(iii)] $\int_X {\tilde r}_t \o_{{\tilde r}_t}^n \leq C, -C \leq \int_X {\tilde r}_t \o^n$ hence also $-C \leq \sup_X {\tilde r}_t$ for some $C > 1$.
\item[(iv)] $-\inf_X {\tilde r}_t \leq C\sup_X {\tilde r}_t + D $ for some $ C,D >0$.
\item[(v)] if $\alpha \in (0,1)$ then $ -\log\Big(\int_X e^{-\alpha({\tilde r}_t - \sup_X{{\tilde r}_t})}\o^n\Big) \leq ((1-\alpha)n -\alpha) \sup_X{{\tilde r}_t} + {C}$ for some $C > 1$.

\item[(vi)] $\sup_X {\tilde r}_t - AM({\tilde r}_t) \geq  \sup_X {\tilde r}_t/C - C \geq  (AM({\tilde r}_t) - \inf_X {\tilde r}_t)/D - D$ for some $C,D > 1$.
\end{itemize}
\end{proposition}
\begin{proof} The estimates in (i) are essentially due to Perelman \cite{st,tz}. The estimates from $(ii)$ are also well known. In fact, one can prove that $t \to AM(u_t)$ is increasing \cite{ct, l}. We the recall  the argument from \cite{r1}. First we notice that
$$-\log\int_X e^{-{\tilde r}_t +f_\o}\o^n = -\log\int_X e^{-\dot {\tilde r}_t}\o^n_{{\tilde r}_t},$$
hence this quantity is uniformly bounded by $(i)$. It is well known that $t \to \mathcal F({\tilde r}_t)$ is decreasing and now looking at the expression of $\mathcal F({\tilde r}_t)$ from \eqref{Dingdef}, we conclude that there exists $C > 1$ such that $AM({\tilde r}_t) \geq -C$. The second estimate of $(ii)$ now follows from the next well known inequality:
\begin{flalign*}
AM({\tilde r}_t) = &\frac{1}{(n+1)\textup{Vol}(X)} \sum_{j=0}^n \int_X {\tilde r}_t \o^{j}\wedge \o_{{\tilde r}_t}^{n-j}\\
\leq& \frac{1}{(n+1)\textup{Vol}(X)}\Big(\int_X {\tilde r}_t \o^n_{{\tilde r}_t} + n\int_X {\tilde r}_t \o^n\Big).
\end{flalign*}

We now prove the estimate of $(iii)$. From \eqref{RicciflowEqScalar} we have $\int_X e^{{\tilde r}_t}\o_{{\tilde r}_t}^n = \int_X e^{\dot {\tilde r}_t + f }\o^n$. Hence the estimates of (i) yield that $\int_X e^{{\tilde r}_t}\o_{{\tilde r}_t}^n$ is uniformly bounded. The first estimate now follows from Jensen's inequality:
$$\frac{1}{\textup{Vol}(X)}\int_X {\tilde r}_t \o_{{\tilde r}_t}^n \leq \log \Big(\frac{1}{\textup{Vol}(X)} \int_X e^{{\tilde r}_t} \o_{{\tilde r}_t}^n\Big).$$
The second and third estimate of $(iii)$ follows now from $(ii)$. 
Estimate $(iv)$ is just the Harnack estimate for the K\"ahler-Ricci flow. For a summary of the proof  we refer to steps $(i)$ and $(iii)$ in the proof of \cite[Theorem 1.3]{r1}, which in turn follows the arguments in \cite{t1}.

We justify the estimate of $(v)$ and the roots of our argument are again from \cite{r1}. To start, we notice that using equation \eqref{RicciflowEqScalar} we can write
\begin{flalign*}
-\log\Big(\int_X e^{-\alpha {\tilde r}_t}\o^n\Big)&=-\log\Big(\int_X e^{-\alpha {\tilde r}_t + {\tilde r}_t -f_\o + \dot {\tilde r}_t}\o^n_{{\tilde r}_t}\Big)\\
&\leq  \frac{1}{\textup{Vol}(X)}\int_X (\alpha-1) {\tilde r}_t \o^n_{{\tilde r}_t} + C\\
&\leq \frac{n(1-\alpha)}{\textup{Vol}(X)}\int_X {\tilde r}_t \o^n +C,
\end{flalign*}
where in the second line we have used the estimates of $(i)$ and $(ii)$. It is well known that there exist $D(\o)>0$ such that
$$D \geq \sup_X u - \int_X u \o^n \geq 0, \ u \in \textup{PSH}(X,\o).$$ Putting together the last two estimates finishes the proof of $(v)$.

Now we turn to the proof of the double estimate in $(vi)$. From the definition of $AM$ and $(iii)$ it follows that
\begin{flalign*}
\sup_X {\tilde r}_t - AM({\tilde r}_t) &\geq \frac{1}{n+1}(\sup_X {\tilde r}_t - \frac{1}{\textup{Vol}(X)}\int_X {\tilde r}_t \o_{{\tilde r}_t}^n) \\
&\geq \frac{1}{n+1}\sup_X {\tilde r}_t - C,
\end{flalign*}
and this establishes the first estimate. The second estimate follows from $(iv)$ and the simple fact that $\sup_X {\tilde r}_t \geq AM({\tilde r}_t)$.
\end{proof}

Finally, we are ready to write down the main result of this paragraph, which phrases some of the estimates from the previous proposition for $AM$--normalized K\"ahler--Ricci trajectories.
\begin{proposition} \label{AmNormEst} Suppose $t \to r_t$ is an an $AM$--normalized K\"ahler--Ricci trajectory.  Let $t \to {\tilde r}_t$ be the corresponding K\"ahler--Ricci trajectory  normalized according to \eqref{RicciflowEqScalar}, i.e. $r_t = {\tilde r}_t - AM({\tilde r}_t)$. For $t \geq0$ the following hold:
\begin{itemize}
\item[(i)]$-\inf_X r_t \leq C \sup_X r_t + C$, for some $C >1$.
\item[(ii)] $\sup_X {\tilde r}_t \leq C \sup_X r_t + C \leq D \sup_X {\tilde r}_t + E,$ for some $C,D,E > 1$.
\item[(iii)] For any $p \geq 1$ we have $\sup_X r_t/C -C \leq d_p(r_0,r_t) \leq C \sup_X r_t + C$ for some $C > 1$.
\item[(iv)] If $\alpha \in (n/(n+1),1)$ and $p \geq 1$ then
$ -\log\Big(\int_X e^{-\alpha(r_t - \sup_X(r_t-r_0)) + f_\o}\o^n\Big) \leq -\varepsilon d_p(r_0,r_t) + {C}$ for some $C > 1$ and $\varepsilon >0$.

\end{itemize}
\end{proposition}
\begin{proof} The estimate in $(i)$ follows from part $(vi)$ of the previous proposition. This last estimate also gives the first estimate of $(ii)$. Estimate $(ii)$ in the previous result immediately gives the second part of $(ii)$.

The first estimate of $(iii)$ is just \cite[Corollary 4]{da4}. By \eqref{mabenergeqv} we have that $d_p(r_0,r_t) \leq \textup{osc}_X (r_0 - r_t)$. Part $(i)$ now implies the second estimate of $(iii)$.

Notice that $\alpha > n/(n+1)$ is equivalent with $(1-\alpha)n - \alpha < 0$ and that the left hand side of $(v)$ is invariant under different normalizations. The estimate of $(iv)$ now follows after we put together parts $(v)$ of the previous proposition with what we proved so far in this proposition.
\end{proof}

\section{Proof of the Main Results}

First we give a proof for Theorem \ref{RicciBounded}. As it turns out, the argument is about putting together the pieces developed in the preceding section.

\begin{theorem} Suppose $(X,J,\o)$ is a Fano manifold and $p \geq 1$. There exists a K\"ahler--Einstein metric in $\mathcal H$ if and only if every K\"ahler--Ricci trajectory $[0,\infty) \ni t \to r_t \in \mathcal H_{AM}$ is $d_p$--bounded. More precisely, the $C^0$--bound along the flow is equivalent to the $d_p$--bound:
\begin{equation}\label{C0mabeqv}
\frac{1}{C}d_p(r_0,r_t)-C\leq \sup_X |r_t| \leq C d_p(r_0,r_t)+C,
\end{equation}
for some $C(p)>1$.
\end{theorem}
\begin{proof} If there exists a K\"ahler--Einstein metric in the cohomology class of $\o$ then by \cite[Theorem 6]{da4} we have that any K\"ahler--Ricci trajectory $d_p$--converges to one such metric, hence stays $d_p$--bounded.

For the other direction, suppose $d_p(0,r_t)$ is bounded. By Proposition \ref{AmNormEst}(ii)(iii), $d_p(0,r_t)$ controls both $\sup_X \tilde r_t$ and $\sup_X r_t$, which in turn control $\| \tilde r_t\|_{L^\infty}$ and $\| r_t\|_{L^\infty}$, by Proposition \ref{AmNormEst}(i) and Proposition \ref{R2NormEst}(iv) respectively. The regularity theory for the K\"ahler--Ricci flow implies now that $t \to \tilde r_t$ converges exponentially fast in any $C^k$ norm to a K\"ahler--Einstein metric, hence so does $t \to r_t$.
\end{proof}

It is well known that the Mabuchi K--energy decreases along K\"ahler--Ricci trajectories. The estimates of the previous section imply that in case $(X,J,\o)$ does not admit a K\"ahler--Einstein metric, any $AM$--normalized K\"ahler--Ricci trajectory $t \to r_t$ satisfies the assumptions of the following theorem:

\begin{theorem}\label{curvegeod} Suppose $[0,\infty) \ni t \to c_t \in \mathcal H_{AM}$ is a curve for which there exists $t_l \to \infty$ satisfying the following properties:
\begin{itemize}
\item[(i)]\textup{(Harnack estimate)} $-\inf_X c_{t_l} \leq C \sup_X c_{t_l} + C$ for some $C>0$.
\item[(ii)]\textup{($C^0$ blow-up)}  $\lim_{l \to \infty} \sup_X c_{t_l} = +\infty$.
\item[(iii)]\textup{(bounded K-energy 'slope')}
$$\limsup_{l \to \infty} \frac{\mathcal M(c_{t_l}) - \mathcal M(c_0)}{\sup_X c_{t_l}} < +\infty.$$
\end{itemize}
Then there exists a curve $[0,\infty) \ni t \to u_t \in \mathcal H_{0,AM}$ which is a non--trivial $d_p$--geodesic ray weakly asymptotic to $t \to c_t$ for all $p \geq 1$.

\end{theorem}

\begin{proof} The idea of the proof is to construct a $d_2$--geodesic ray satisfying all the necessary properties. At the end we will conclude that this same curve is also a $d_p$--geodesic ray for any $p \geq 1$.

By setting $\o := \o + i \ddbar c_0$ and $c_t := c_t - c_0$, we can assume without loss of generality that $c_0=0$. As $\inf_X c_{t_l} \leq C \sup_X c_{t_l} +C$, the same argument as in the previous theorem gives:
\begin{equation}\label{c_tMabC0est}
\frac{1}{C}d_p(0,c_{t_l})-C\leq \sup_X |c_{t_l}| \leq C d_p(0,c_{t_l})+C,
\end{equation}
for any $p \geq 1$. The fact that $\lim_l \sup_X c_{t_l} = +\infty$ implies now that $f_l={d_2(0,c_{t_l})} \to \infty$. Let
$$[0,f_l] \ni t \to u_t^l \in \mathcal H_\Delta$$ be the unit speed (re--scaled) weak geodesic curve of \eqref{BVPGeod}, joining $c_0 = 0$ with $c_{t_l}$. By our choice of normalization it follows that
\begin{equation}\label{normalize}
AM(u^l_t) = 0  \ \textup{ and } \ d_2(0, u^l_t) = t, \ t \in [0,f_l].
\end{equation}
By our assumptions and \eqref{c_tMabC0est}  there exists $C,D > 1$ such that
$$-C d_2(0,c_{t_l}) - C \leq -D\sup_X c_{t_l}-D\leq \inf_X c_{t_l} \leq \sup_X c_{t_l} \leq C d_2(0,c_{t_l})+C.$$
Rewriting this, as $u^l_{f_l}=c_{t_l}$ and $\sup_X c_{t_l} \to \infty$, for $l$ big enough we obtain:
\begin{equation}\label{endpointest}
- C \leq \frac{-D'\sup_X u_{t_l}^l}{f_l}\leq \frac{\inf_X u_{f_l}^l}{f_l} \leq \frac{\sup_X u_{f_l}^l}{f_l}\leq C,
\end{equation}
As $u^l_0 =0$ for all $l$, using \eqref{endpointest} and Theorem \ref{m_norm_thm} we can conclude that
\begin{equation}
\label{supest}
- C \leq \frac{-D'\sup_X u_{t}^l}{t}\leq \frac{\inf_X u_{t}^l}{t} \leq \frac{\sup_X u_{t}^l}{t}\leq C, \ t \in [0,f_l].
\end{equation}
By the main result of  \cite{bb} (for a different approach see \cite{clp}) it follows that the map $t \to \mathcal M(u^l_t)$ is convex and non-positive. In particular, for $t \geq 0$ we have:
$$\frac{H_{\o}(\o_{u^l_t})-L(u^l_t)}{t} = \frac{\mathcal M (u^l_t) - \mathcal M (u_0)}{t}\leq \frac{\mathcal M (c_{t_{l}}) - \mathcal M (c_0)}{f_l}\leq C <\infty,$$
where in the last estimate we have used condition (ii) in the statement of the theorem along with \eqref{c_tMabC0est}. Proposition \ref{LEst} now implies that there exists $C>1$ such that
\begin{equation}\label{entropybound}
0 \leq H_{\o}(\o_{u^l_t})\leq L(u^l_t)+Et\leq Cd_2(0,u^l_t)+Dt= (C+D) t.
\end{equation}
Fix now $s \geq 0$.  From \eqref{supest} and \eqref{entropybound} it follows using Theorem \ref{boundedcomp} that there exists $l'_k \to \infty$ and $u_s \in \mathcal H_0$ such that $d_2(u^{l'_k}_s,u_s) \to 0$. As $AM$ is continuous with respect to $d_2$, by \eqref{normalize} we also have
$AM(u_s)=0$ and $d_2(0,u_s)=s$.

Building on this last observation, using a Cantor type diagonal argument, we can find sequence $l_k \to \infty$ such that for each $h \in \Bbb Q_+$ there exists $u_h \in \mathcal H_0$ satisfying $d_p(u^{l_k}_h,u_h) \to 0$,  $AM(u_h)=0$ and $d_2(0,u_h)=h$.

As $t \to u^l_t$ are unit speed $d_2$--geodesic segments, for any $a,b,c \in \Bbb Q_+$ satisfying $a < b < c$ we have
$$d_2(u^{l_k}_a,u^{l_k}_b) + d_2(u^{l_k}_b,u^{l_k}_c)=c-a =d_2(u^{l_k}_a,u^{l_k}_c).$$
Taking the limit $l_k \to \infty$ we will also have
$$d_2(u_a,u_b) + d_2(u_b,u_c)=c-a =d_2(u_a,u_c).$$
Hence, by density we can extend $h \to u_h$ to a unit speed $d_2$--geodesic $[0,\infty) \ni t \to u_t \in \mathcal H_{0,AM}$ weakly asymptotic to $t \to c_t$. This $d_2$--geodesic is non-trivial, i.e. not of the form $u_t = u_0 + ct$ for some $c \in \Bbb R$. Indeed, this would contradict the fact $AM(u_t)=0$ and $t\to u_t$ is unit speed with respect to $d_2$.

Finally, as $t \to u_t$ is a bounded $d_2$--geodesic ray, Proposition \ref{samegeod} says that $t \to u_t$ is a $d_p$ geodesic ray as well for any $p \geq 1$.
\end{proof}

When the curve $t \to c_t$ in the previous theorem is a diverging K\"ahler-Ricci trajectory, the weakly asymptotic ray produced by the previous theorem has additional properties:

\begin{theorem} \label{alphaest} Suppose $(X,J,\o)$ is a Fano manifold without a K\"ahler--Einstein metric in $\mathcal H$ and $[0,\infty) \ni t \to r_t \in \mathcal H_{AM}$ is a K\"ahler-Ricci trajectory. Let $t \to u_t$ be the geodesic ray produced by the previous theorem. The following holds:
\begin{itemize}
\item[(i)]The map $t \to \mathcal F(u_t)$ is decreasing. If additionally $(X,J)$ does not admit non--trivial holomorphic vector fields then $t \to \mathcal F(u_t)$ is strictly decreasing.
\item[(ii)]The "sup-normalized" potentials $u_t - \sup_X (u_t - u_0) \in \mathcal H_0$ decrease pointwise to $u_\infty \in \textup{PSH}(X,\o)$ for which $\int_X e^{-\frac{n}{n+1}u_\infty}\o^n =\infty$.
\end{itemize}
\end{theorem}
\begin{proof}We work with the notations of the previous theorem. To show $t\to \mathcal F(u_t)$ is decreasing, we claim first that for any $t>0$, $\mathcal F(u_0)\geq \mathcal F(u_t)$. It is well known that $t \to \mathcal F(r_t)$ is decreasing, hence $\mathcal F(u_0)\geq \mathcal F(u^l_{f_l})$. By Berndtsson's theorem \cite{brn1}, the maps $t \to \mathcal F(u^l_t)$ are convex, hence we also have
$$\mathcal F(u_0)\geq \mathcal F(u^l_t), \ t\in [0, f_l].$$
As noted earlier, the maps  $\mathcal F$ are continuous with respect to $d_2$. By passing to the limit, the claim is proved. As $t\rightarrow \mathcal F(u_t)$ is convex and $\mathcal F(u_0)\geq \mathcal F(u_t)$ for any $t\in (0, \infty)$, $\mathcal F$ has to be decreasing.

If additionally $(X,J)$ does not admit non--trivial holomorphic vector fields then $t \to \mathcal F(u_t)$ is strictly decreasing. Indeed, if this were not the case, then there would exist $t_0 \geq 0$ such that
$$\frac{\partial }{\partial t}\mathcal F(u_t)=0, \ t \geq t_0.$$
By the second part of Berndtsson's convexity theorem \cite{brn1}, this implies that $(X,J)$ admits a non--trivial holomorphic vector field, which is a contradiction.

We turn to part $(ii)$. For $n/(n+1) < \alpha < 1$ each curve $t \to \alpha u^l_t$ is a subgeodesic, hence it follows from \cite{brn1} that each map
$$t \to -\log\Big(\int_X e^{-\alpha u^l_t + f_\o}\o^n\Big)$$
is convex. As $u^l_0\equiv 0$, by Theorem \ref{m_norm_thm} the map $t \to \sup_X u^l_t$ is linear, hence the function
\begin{flalign*}
t \to \mathcal G_\alpha(u^l_t)&=-\log\Big(\int_X e^{-\alpha (u^l_t  - \sup_X u^l_t)+ f_\o}\o^n\Big)\\
&=-\log\Big(\int_X e^{-\alpha u^l_t+ f_\o}\o^n\Big) - \alpha\sup_X u^l_t
\end{flalign*}
is also convex. By theorem \ref{AmNormEst}(iv) this implies that $\mathcal G_\alpha(u^l_t) \leq - \varepsilon d_2(0,u^l_t) + C=-\varepsilon t + C.$
Similarly to $\mathcal F(\cdot)$, the functional $\mathcal G_\alpha(\cdot)$ is also continuous with respect to $d_2$, hence by taking the limit $l_k \to \infty$ in this last estimate we obtain:
\begin{equation}\label{Gest}
\mathcal G_\alpha(u_t) \leq - \varepsilon t + C.
\end{equation}
As discussed after Theorem \ref{m_norm_thm}, the decreasing limit $u_\infty =\lim_{t \to\infty}(u_t - \sup_X u_t)$ is a well defined and not identically equal to $-\infty$.
Letting $t \to \infty $ in \eqref{Gest} by the monotone convergence theorem we obtain that $\int_X e^{-\alpha u_\infty}\o^n = \infty$. As $n/(n+1) < \alpha < 1$ is arbitrary, the recent resolution of the openness conjecture (see \cite{brn2,gzh}) implies part (ii).
\end{proof}
We believe $t\rightarrow \mathcal F(u_t)$ should be strictly decreasing even if $X$ has holomorphic vector fields. We can show this when the Futaki invariant is nonzero as we elaborate below. Note that along the K\"ahler-Ricci trajectory $t \to r_t$ the $\mathcal F$--functional is strictly decreasing unless the initial metric is K\"ahler--Einstein. Using the identity $$\frac{e^{-r_t+f_{\o}}}{\int_Xe^{-r_t+f_{\o}}\o^n}\o^n = e^{f_{\o_{r_t}}}\o^n_{r_t}$$ we can write
\[
\frac{\partial \mathcal F(r_t)}{\partial t}=-\int_X f_{\omega_{r_t}} (e^{f_{\o_{r_t}}}-1)\omega^n_{r_t}.
\]
It is natural to introduce the following quanitity:
\[
\epsilon(\omega)=\inf_{u \in \mathcal H}\int_X f_{\omega_u} (e^{f_{\omega_u}}-1)\omega^n_u\geq 0.
\]
This quanitity is clearly an invariant of $(X, J, [\omega])$. If $\epsilon(\omega)>0$, then there exists no K\"ahler-Einstein metric in $\mathcal H$. By Jensen's inequality, for any $u \in \mathcal H$ we have $\int_M f_{\omega_u} \omega^n_u\leq 0$, hence we can write
$$
\int_M f_{\omega_u} (e^{f_{\omega_u}}-1)\omega^n_{u}\geq \int_M f_\omega e^{f_{\omega_u}} \omega^n_u.
$$
By \cite{h2}, the right hand side above (defined as the $H$-functional) is nonnegative and is uniformly bounded away from zero if the Futaki invariant is nonzero, implying in this last case the bound $\epsilon(\omega) >0$. Finally, we note the following result:

\begin{proposition}\label{FutakiStrict}Suppose $t \to r_t$ and $t \to u_t$ are as in the previous theorem. If $\epsilon(\omega)>0$, then the map $t\rightarrow \mathcal F(u_t)$ is strictly decreasing. More precisely, there exists $C>0$ such that
$\mathcal F(u_t) \leq \mathcal F(u_0)- Ct, \ t \geq 0$.
\end{proposition}

\begin{proof}By the discussion above, we have the estimate
$$\mathcal F(r_{t_l}) -\mathcal F(r_0) \leq -\epsilon(\o) t_l.$$
Using the notation of the previous theorem's proof, by the estimates of paragraph 2.2, there exists $C,C'>0$ such that for $l$ big enough:
\[
f_l=d_2(0, r_{t_l})\leq C' \sup_X r_{t_l} \leq C t_l.
\]
From our observations it follows that
\[
\frac{\mathcal F(u^l_{f_l}) - \mathcal F(u_0)}{f_l}=\frac{\mathcal F(r_{t_l}) - \mathcal F(r_0)}{f_l} \leq -\frac{\epsilon(\o)}{C}.
\]
By the convexity of $\mathcal F$ we can conclude that
\[\frac{\mathcal F(u^l_t) - \mathcal F(u_0)}{t}\leq -\frac{\epsilon(\o)}{C}, \ t\in (0, f_l].\]
Letting $ l \to \infty$ we obtain
\[\frac{\mathcal F (u_t) - \mathcal F(u_0)}{t}\leq -\frac{\epsilon(\o)}{C}, \ t\in (0,\infty).\]
\end{proof}

Finally we prove the equivalence of geodesic stability and existence of K\"ahler--Einstein metrics:

\begin{theorem} \label{raystability}Suppose $p \in \{1,2 \}$ and $(X,J,\o)$ is a Fano manifold without non--trivial holomorphic vector--fields and $u \in \mathcal H$. There exists no K\"ahler-Einstein metric in $\mathcal H$ if and only if for any $u \in \mathcal H$ there exists a $d_p$--geodesic ray $[0,\infty) \ni t \to u_t \in \mathcal H_{0,AM}$ with $u_0=u$ such that the function $t \to \mathcal F(u_t)$ is strictly decreasing.
\end{theorem}

\begin{proof} The only if direction is a consequence of the previous theorem. We argue the if direction. Suppose there exists a K\"ahler--Einstein metric in $\mathcal H$.

In case $p=1$ it is enough to invoke \cite[Theorem 6]{da4}. Indeed, this result says that on a Fano manifold without non--trivial holomorphic vector--fields existence of a K\"ahler-Einstein metric in $\mathcal H$ is equivalent to the $d_1$--properness of $\mathcal F$ (sublevel sets of $\mathcal F$ are $d_1$--bounded). Hence the map $t \to \mathcal F(u_t)$ can not be bounded for any $d_1$--geodesic ray $t \to u_t$.

For the case $p=2$ we first claim that $d_2$--geodesic rays are also $d_1$--geodesic rays. Indeed, this follows from the $CAT(0)$ property of $\overline{(\mathcal H,d_2)}=(\mathcal E^2(X,\o),d_2)$ \cite{da3,cc}, as we argue now. Because of this property, $d_2$--geodesic segments connecting different points of $\mathcal H_0$ are unique, hence they are always of the type described in \eqref{EpGeodDef}, which are also $d_1$--geodesics (Proposition \ref{samegeod}). Clearly, the same statement holds for geodesic rays as well, not just segments, proving the claim. Now we can use \cite[Theorem 6]{da4} again to conclude the argument.
\end{proof}

Using the convexity of $\mathcal F$ along $d_2$--geodesics, the above proof additionally shows that there exists a K\"ahler--Einstein metric on $(X,J,\o)$ if and only if $t \to \mathcal F(u_t)$ is eventually strictly increasing for all $d_2$--geodesic rays $t \to u_t$.

\paragraph{Acknowledgments.} The first author would like to thank Yanir Rubinstein for numerous stimulating conversations related to the topic of the paper and for L\'aszl\'o Lempert for suggestions on how to improve the presentation. 

\end{document}